\documentclass{amsart}
\usepackage{amssymb}
\usepackage{amsmath}
\usepackage{amsthm}
\usepackage{amsfonts}
\usepackage{amsxtra}

\newtheorem{theorem}{Theorem}[section]
\newtheorem{corollary}[theorem]{Corollary}

\newtheorem{definition}[theorem]{Definition}
\newtheorem{proposition}[theorem]{Proposition}

\theoremstyle{remark}
\newtheorem{remark}[theorem]{Remark}
\newtheorem{example}[theorem]{Example}

\usepackage{tikz-cd}
\usepackage{extarrows}
\usepackage{hyperref}
\usepackage[all]{xy}

\newcommand{\coTHH}{\mathrm{coTHH}}
\newcommand{\THH}{\mathrm{THH}}
\newcommand{\coHH}{\mathrm{coHH}}
\newcommand{\HH}{\mathrm{HH}}
\newcommand{\Tor}{\mathrm{Tor}}
\DeclareMathOperator{\Tot}{Tot}

\begin{document}

\title{Computations of Relative Topological CoHochschild Homology} 
\author{Sarah Klanderman}             
\email{sklanderman@marian.edu}
\address{Department of Mathematics,
         Marian University,
         3200 Cold Spring Road,
         Indianapolis, IN, 46222,
         USA}
\date{August 9, 2021}

\begin{abstract}
{Hess and Shipley defined an invariant of coalgebra spectra called topological coHochschild homology, and Bohmann-Gerhardt-H{\o}genhaven-Ship\-ley-Ziegenhagen developed a coB\"okstedt spectral sequence to compute the homology of $\coTHH$ for coalgebras over the sphere spectrum. We construct a relative coB\"okstedt spectral sequence to study $\coTHH$ of coalgebra spectra over any commutative ring spectrum $R$.  
Further, we use algebraic structures in this spectral sequence to complete some calculations of the homotopy groups of relative topological coHochschild homology.}
\end{abstract}

\maketitle

 \section{Introduction}
We develop computational tools for studying topological coHochschild homology ($\coTHH$). Recent work of Hess and Shipley \cite{hess2018invariance} defines this invariant to study coalgebra spectra. Work of Malkiewich \cite{malkiewich2017cyclotomic} and Hess-Shipley \cite{hess2018invariance} shows $\coTHH$ of suspension spectra is related to suspension spectra of free loop spaces for simply connected spaces $X$. In particular,
\[
\coTHH(\Sigma^\infty_+ X) \simeq \Sigma^\infty_+ \mathcal{L}X \simeq \THH(\Sigma^\infty_+(\Omega X)),
\]
where the last equivalence comes from work of B\"okstedt and Waldhausen \cite{bokstedtwaldhausen}.  Thus $\coTHH$ is relevant for studying the homology of free loop spaces, $\mathcal{L}X$, the main topic of the field of string topology \cite{chas1999string, cohen2003loop}. In fact, Bohmann-Gerhardt-Shipley use topological coHochschild homology to compute the homology of free loop spaces in \cite{bohmanngerhardtshipley}. Further, because $\THH(\Sigma^\infty_+(\Omega X))$ is directly related to the algebraic $K$-theory of the space $X$ via trace methods, $\coTHH$ also has applications for algebraic $K$-theory of spaces.

In 2018, Bohmann-Gerhardt-H{\o}genhaven-Shipley-Ziegenhagen showed that there is a \textit{coB\"okstedt spectral sequence} that is the Bousfield-Kan spectral sequence for the cosimplicial spectrum $\coTHH(C)^\bullet$, for $C$ a coalgebra spectrum over the sphere spectrum \cite{bohmann2018computational}. This spectral sequence has classical coHochschild homology of \cite{doi1981homological} as its $E_2$-term, and in cases when it does indeed converge we have:
\[
E_2^{*,*}= \coHH_*(H_*(C;k)) \implies H_*(\coTHH_*(C);k).
\]
In their work however, these tools are only set up to study coalgebra spectra over the sphere spectrum, $\mathbb{S}$. Examples of this sort are closely related to suspension spectra of spaces, and recent work of P\'eroux-Shipley shows that such examples are fairly limited \cite{P_roux_2019}.  In this paper, we broaden these tools to apply to $\coTHH$ for coalgebras over any commutative ring spectrum.

We call the spectral sequence that allows us to study the topological coHochschild homology of coalgebras over an arbitrary commutative ring spectrum $R$ the \textit{relative coB\"okstedt spectral sequence}$:$ 
\begin{theorem}
Let $E$ and $R$ be commutative ring spectra, $C$ an $R$-coalgebra spectrum that is cofibrant as an $R$-module, and $N$ a $(C,C)$-bicomodule spectrum. If $E_*(C)$ is flat over $E_*(R)$, then there exists a Bousfield-Kan spectral sequence for the cosimplicial spectrum $\coTHH^R(N,C)^\bullet$ that abuts to $E_{t-s}(\coTHH^R(N,C))$ with $E_2$-page
\[
E_2^{s,t} = \coHH^{E_*(R)}_{s,t}(E_*(N), E_*(C))
\]
given by the classical coHochschild homology of $E_*(C)$ with coefficients in $E_*(N)$.
\end{theorem}
Further, we identify conditions for convergence of this spectral sequence.  In particular, for the case when $E = \mathbb{S}$, if for every $s$ there exists some $r$ so that $E^{s,s+i}_r = E^{s,s+i}_\infty$ then the relative coB\"okstedt spectral sequence converges completely to $\pi_*(\coTHH^R(N,C))$.

By work of Angeltveit-Rognes, the classical B\"okstedt spectral sequence for a commutative ring spectrum has the structure of a spectral sequence of Hopf algebras under certain flatness conditions \cite{angeltveit2005hopf}.  Bohmann-Gerhardt-Shipley show that under appropriate coflatness conditions, the coB\"okstedt spectral sequence for a cocommutative coalgebra spectrum has what is called a \textit{$\square$-Hopf algebra structure}, an analog of a Hopf algebra structure for working over a coalgebra \cite{bohmanngerhardtshipley}. We show that the relative coB\"okstedt spectral sequence has a similar structure and use it to prove the following results.

\begin{theorem}
For a field $k$, let $C$ be a cocommutative $Hk$-coalgebra spectrum that is cofibrant as an $Hk$-module with $\pi_*(C) \cong \Lambda_k(y)$ for $|y|$ odd and greater than 1. Then the relative coB\"okstedt spectral sequence collapses and 
\[
\pi_*(\coTHH^{Hk}(C)) \cong \Lambda_k(y) \otimes k[w]
\]
as graded $k$-modules for $|w|=|y|-1$.
\end{theorem}

\begin{theorem}
Let $k$ be a field and let $p=char(k)$, including 0. For $C$ a cocommutative $Hk$-coalgebra spectrum that is cofibrant as an $Hk$-module with $\pi_*(C) \cong \Lambda_k(y_1, y_2)$ for $|y_1|, |y_2|$ both odd and greater than 1, if $p^m$ is not equal to $\frac{|y_2|-1}{|y_1|-1}+1$ for all $m \ge 1$ and $p \ne 2$,
then the relative coB\"okstedt spectral sequence collapses and
\[
\pi_*(\coTHH^{Hk}(C)) \cong \Lambda_k(y_1, y_2) \otimes k[w_1, w_2],
\]
as graded $k$-modules for $|w_i| = |y_i|-1$. Moreover, the same result holds at the prime $p=2$ with the additional assumption that $p^m \ne \frac{|y_2|-1}{|y_1|-1}$ for $m \ge 1$.
\end{theorem}

\noindent Further, in a result analogous to the work of Bohmann-Gerhardt-H{\o}genhaven-Shipley-Ziegenhagen \cite{bohmann2018computational} we have

\begin{theorem}
Let $k$ be a field and let $C$ be a cocommutative $Hk$-coalgebra spectrum that is cofibrant as an $Hk$-module spectrum,
and whose homotopy coalgebra is
\[
\pi_*(C) = \Gamma_k[x_1, x_2, \ldots],
\]
where the $x_i$ are in non-negative even degrees and there are only finitely many of them in each degree.  Then the relative coB\"okstedt spectral sequence  calculating the homotopy groups of the topological coHochschild homology of $C$ collapses at $E_2$, and
\[
\pi_*(\coTHH^{Hk}(C)) \cong \Gamma_k[x_1, x_2, \ldots] \otimes \Lambda_k(z_1, z_2, \ldots)
\]
as $k$-modules, with $z_i$ in degree $|x_i|-1$.
\end{theorem}

\subsection{Organization}
The paper is organized as follows.  Section 2 provides background on coalgebras in spectra and (topological) coHochschild homology. In Section 3 we construct the relative coB\"okstedt spectral sequence.  Section 4 examines the algebraic structures of this spectral sequence based on work of \cite{bohmanngerhardtshipley}, and Section 5 discusses some explicit calculations of topological coHochschild homology. 

\subsection{Acknowledgements}
These results are part of the author's dissertation work, and as such the author would like to thank her advisor, Teena Gerhardt, for her help and guidance over the years. In addition, discussions with Maximilien P\'eroux and \"Ozg\"ur Bay\i nd\i r about their work with coalgebras and conversations with Gabe Angelini-Knoll were particularly helpful.

\section{Topological coHochschild homology}

We recall the definition of a coalgebra for the general setting of a symmetric monoidal category in order to introduce coalgebras in spectra and consider examples.

\begin{definition}
    Let $(\mathcal{D}, \otimes, 1)$ be a symmetric monoidal category. Then a (coassociative, counital) \textbf{coalgebra} $C \in \mathcal{D}$ has a comultiplication $\Delta: C \to C \otimes C$ that is coassociative and counital, i.e. there exists a counit morphism $\epsilon: C \to 1$ such that the following coassociativity and counitality diagrams commute:
    
    \begin{minipage}{.45\textwidth}
    {\large
    \[ \begin{tikzcd}[ampersand replacement=\&]
    C \arrow{r}{\Delta} \arrow[swap]{d}{\Delta} \& C \otimes C \arrow{d}{\text{Id } \otimes \Delta} \\%
    C \otimes C \arrow{r}{\Delta \otimes \text{Id }} \& C \otimes C \otimes C
    \end{tikzcd}
    \]}
    \end{minipage}
    \begin{minipage}{.4\textwidth}
    {\large
    \[ \begin{tikzcd}[ampersand replacement=\&]
    C \arrow{r}{\Delta} \arrow[swap]{d}{\Delta} \arrow[swap]{dr}{\text{Id}} \& C \otimes C \arrow{d}{\text{Id } \otimes \epsilon} \\%
    C \otimes C \arrow{r}{\epsilon \otimes \text{Id }} \& 1 \otimes C \cong C \cong C \otimes 1
    \end{tikzcd}
    \]}
    \end{minipage}
\end{definition}

Examples of coalgebras over a field include the polynomial, exterior, and divided power coalgebras. 

\begin{example}
For a field $k$, the polynomial coalgebra $k[w_1, w_2, \ldots]$ for $w_i$ in even degree is the vector space with basis given by $\{w_i^j \}$ for $j \ge 0$ and $i \ge 1$. It has coproduct
\begin{align*}
    \Delta(w_i^j) &= \sum_k \binom{j}{k} w_i^k \otimes w_i^{j-k}
\end{align*}
and counit 
\begin{align*}
    \epsilon (w_i^j) &=  \begin{cases} 1 & \text{if } j=0\\ 0 & \text{if } j \ne 0. \end{cases}
\end{align*}
\end{example}

\begin{example}
For a field $k$, the exterior coalgebra $\Lambda_k(y_1, y_2, \ldots)$ for $y_i$ in odd degrees is the vector space with basis given by $\{1, y_i \}$ for $i \ge 1$, which has coproduct 
\begin{align*}
    \Delta(y_i) &= 1 \otimes y_i + y_i \otimes 1\\
    \Delta(1) &= 1 \otimes 1
\end{align*}
and counit 
\begin{align*}
    \epsilon(y_i) &= 0\\
    \epsilon(1)&=1
\end{align*}
\end{example}

\begin{example}
For a field $k$, the divided power coalgebra $\Gamma_k[x_1, x_2, \ldots]$ with $x_i$ in even degrees is the vector space with basis given by $\{ \gamma_j(x_i) \}$ for $j \ge 0$ and $i \ge 1$.  It has coproduct
\begin{align*}
    \Delta(\gamma_j (x_i)) &= \sum_{a+b=j} \gamma_a(x_i) \otimes \gamma_b(x_i)
\end{align*}
where $\gamma_0(x_i)= 1, \gamma_1(x_i)=x_i$, and counit 
\begin{align*}
    \epsilon (\gamma_j(x_i)) &= \begin{cases} 1 & \text{if } j=0\\ 0 & \text{if } j \ne 0. \end{cases}
\end{align*}
\end{example}

Because it will also be useful to have the ``dual'' notion of modules on which $C$ \textit{coacts}, we also recall the following definition of a comodule.

\begin{definition}
    Let $R$ be a commutative ring and $C$ an $R$-coalgebra.  Then $N$ is a \textbf{right $C$-comodule} if it is an $R$-module together with an $R$-linear map $\gamma: N \to N \otimes_R C$ that is coassociative and counital, so that the following diagrams commute:
    
    \begin{minipage}{.45\textwidth}
     {\large
    \[ \begin{tikzcd}[ampersand replacement=\&]
    N \arrow{r}{\gamma} \arrow[swap]{d}{\gamma} \& N \otimes_R C \arrow{d}{\text{Id } \otimes \Delta} \\%
    N \otimes_R C \arrow{r}{\gamma \otimes \text{Id }} \& N \otimes_R C \otimes_R C
    \end{tikzcd}
    \]}
    \end{minipage}
    \begin{minipage}{.4\textwidth}
    {\large
    \[ \begin{tikzcd}[ampersand replacement=\&]
    N \arrow{r}{\gamma} \arrow[swap]{dr}{\text{Id}} \& N \otimes_R C \arrow{d}{\text{Id } \otimes \epsilon} \\%
    \& N
    \end{tikzcd}
    \]}
    \end{minipage}
    
    \noindent The map $\gamma$ is referred to as a \textit{right $C$-coaction}. Similarly, a \textbf{left $C$-comodule} is an $R$-module together with an $R$-linear map $\psi: N \to C \otimes_R N$ that is coassociative and counital, so that the following diagrams commute:
    
    \begin{minipage}{.45\textwidth}
     {\large
    \[ \begin{tikzcd}[ampersand replacement=\&]
    N \arrow{r}{\psi} \arrow[swap]{d}{\psi} \& C \otimes_R N \arrow{d}{\Delta \otimes \text{Id }} \\%
    C \otimes_R N \arrow{r}{\text{Id } \otimes \psi} \& C \otimes_R C \otimes_R N
    \end{tikzcd}
    \]}
    \end{minipage}
    \begin{minipage}{.4\textwidth}
    {\large
    \[ \begin{tikzcd}[ampersand replacement=\&]
    N \arrow{r}{\psi} \arrow[swap]{dr}{\text{Id}} \& C \otimes_R N \arrow{d}{\epsilon \otimes \text{Id }} \\%
    \& N
    \end{tikzcd}
    \]}
    \end{minipage}
    
    \noindent The map $\psi$ is referred to as a \textit{left $C$-coaction}.
\end{definition}

The right and left comodule structures together can be used to define a bicomodule:

\begin{definition}
     For $R$-coalgebras $C, D$, a $(C,D)$-\textbf{bicomodule} $N$ is a left $C$-comodule with left coaction $\psi: N \to C \otimes_R N$ and right $D$-comodule with right coaction $\gamma: N \to N \otimes_R D$ that satisfies the following commutative diagram:
     {\large
    \[ \begin{tikzcd}[ampersand replacement=\&]
    N \arrow{r}{\gamma} \arrow[swap]{d}{\psi} \& N \otimes_R D \arrow{d}{\psi \otimes \text{Id }} \\%
    C \otimes_R N \arrow{r}{\text{Id } \otimes \gamma} \& C \otimes_R N \otimes_R D
    \end{tikzcd}
    \]}
\end{definition}

Since we want to apply these definitions in the topological setting, we introduce our notation for coalgebras in spectra in the following example.

\begin{example}
    A \textbf{coalgebra spectrum} is a coalgebra in one of the symmetric monoidal categories of spectra. For a commutative ring spectrum $R$, an \textbf{$R$-coalge\-bra spectrum} $C$ is a coalgebra in the symmetric monoidal category of $R$-modules.  It has comultiplication $\Delta: C \to C \wedge_R C$ and counit $\epsilon: C \to R$, satisfying the coassociativity and counitality conditions.
\end{example}

A familiar example of a coalgebra in spectra comes from suspension spectra, since for a space $X$, the diagonal map $X \to X \times X$ on topological spaces induces a comultiplication map on the suspension spectrum $\Sigma_+^\infty (X)$ making $\Sigma_+^\infty (X)$ into a coalgebra spectrum.
However, most spectra do not have diagonal maps and P\'eroux-Shipley \cite{P_roux_2019} show that examples of this form are quite limited.  In particular, they prove that all coalgebra spectra over $\mathbb{S}$ are cocommutative in monoidal categories of spectra such as symmetric spectra, orthogonal spectra, $\mathbb{S}$-modules, etc. As we saw in the above example, some coalgebras over the sphere spectrum come from suspension spectra; in fact, P\'eroux-Shipley further show that in model categories all $\mathbb{S}$-coalgebras are closely related to suspension spectra. 

This rigid structure of $\mathbb{S}$-coalgebras thus provides motivation for studying other kinds of coalgebra spectra.
Examples of these kinds of coalgebra spectra include $H\mathbb{F}_p$-coalgebras such as 
$H\mathbb{F}_p \wedge_{H\mathbb{Z}} H\mathbb{F}_p$,
which we will examine later on in further detail. We will study these coalgebras via the invariant \textit{topological coHochschild homology}. 
 
Bohmann-Gerhardt-H{\o}genhaven-Shipley-Ziegenhagen \cite{bohmann2018computational} define (topological) coHochschild homology for any general symmetric monoidal category, but classical coHochschild homology as defined by Doi \cite{doi1981homological} can be recovered from this definition by considering the category of coalgebras over a field.

\begin{definition}
Let $(\mathcal{D}, \otimes, 1)$ be a symmetric monoidal model category and let $C \in \mathcal{D}$ be a coalgebra with coassociative comultiplication $\Delta: C \to C \otimes C$ and counit $\epsilon: C \to 1$.  Further, let $N$ be a $(C,C)$-bicomodule with left and right coactions $\psi: N \to C \otimes N$ and $\gamma: N \to N \otimes C$.
Define $\coTHH(N, C)^\bullet$ to be the  cosimplicial object with r-simplices $\coTHH(N, C)^r = N \otimes C^{\otimes r}$,
with coface maps
\begin{align*}
    \delta_i = \begin{cases}
\gamma \otimes \text{Id}^{\otimes r} &i=0\\
\text{Id}^{\otimes i} \otimes \Delta \otimes \text{Id}^{\otimes (r-i)} &0 < i \le r\\
\Tilde{t} \circ (\psi \otimes \text{Id}^{\otimes r}) &i=r+1
\end{cases}
\end{align*}
where $\Tilde{t}$ is the map that twists the first factor to the last, and with codegeneracy maps $\sigma_i: N \otimes C^{\otimes (r+1)} \to N \otimes C^{\otimes r}$ for $0 \le i \le r$
\[
\sigma_i = \text{Id}^{\otimes (i+1)} \otimes \epsilon \otimes \text{Id}^{\otimes r-i}.
\]

This gives a cosimplicial object of the form
\begin{align*}
    &\vdots\\
    N \otimes &C \otimes C\\
    \uparrow \downarrow &\uparrow \downarrow \uparrow\\
    N &\otimes C\\
    \uparrow &\downarrow \uparrow\\
    &N
\end{align*}
The \textbf{topological coHochschild homology} of the coalgebra $C$ with coefficients in $N$ is defined by 
\[
\coTHH(N,C) = \Tot(\mathcal{R} \coTHH(N,C)^\bullet)
\]
where $\mathcal{R}$ is the Reedy fibrant replacement and $\Tot$ represents the totalization.
\end{definition}

The above definition gives topological coHochschild homology with coefficients in a $(C,C)$-bicomodule $N$, but when we consider $C$ as a bicomodule over itself, we write $\coTHH(C)$ for coefficients in $C$.
We will further decorate the notation with $\coTHH^R(C)$ when we consider the topological coHochschild homology of $C$ relative to $R$, i.e. $\coTHH$ of an $R$-coalgebra $C$ for $R$ a commutative ring spectrum.

As mentioned above, working in the category of coalgebras over a field recovers the classical $\coHH$ of Doi \cite{doi1981homological}.  Our general convention will be to specifically refer to this construction as \textit{topological} coHochschild homology when we are considering as input some coalgebra \textit{spectrum}.  For instance, we will refer to the work of Hess-Parent-Scott \cite{hess2009cohochschild} as studying \textit{coHochschild homology} of differential graded coalgebras (dg-coalgebras) over a field. However, these constructions are all coming from the same general framework by work of Bohmann-Gerhardt-H{\o}genhaven-Shipley-Ziegenhagen \cite{bohmann2018computational}.

\section{Construction of a relative coB\"okstedt spectral sequence}

Given the homology theory $E_*$ associated to the commutative ring spectrum $E$, recall the classical B\"okstedt spectral sequence \cite{bokstedt1985topological} for studying the topological Hochschild homology of an $R$-algebra:

\begin{theorem}[{\cite[Thm X.2.9]{elmendorf1995rings}}]
Suppose $E$ and $R$ are commutative ring spectra, 
$A$ is an $R$-algebra, and
$M$ is a cell $(A,A)$-bimodule.  Then if $E_*(A)$
is flat over $E_*(R)$, 
then there exists a spectral sequence
\[
E^2_{p,q} = \HH^{E_*(R)}_{p,q}(E_*(A),E_*(M)) \implies E_{p+q}(\THH^R(A,M))
\]
\end{theorem}

This result follows from the general spectral sequence that arises from the skeletal filtration of the simplicial spectrum $\THH^R(A,M)_\bullet$. Similarly, associated to a cosimplicial spectrum is a Bousfield-Kan spectral sequence \cite{bousfield1972homotopy}, which when applied to the cosimplicial spectrum $\coTHH(C)^\bullet$ yields a spectral sequence Bohmann-Gerhardt-H{\o}genhaven-Shipley-Ziegenhagen \cite{bohmann2018computational} call the \textit{coB\"okstedt spectral sequence}.  They identify its $E_2$-term as the classical coHochschild homology of coalgebras in the sense of Doi \cite{doi1981homological}:
\begin{theorem}[{\cite[Thm 4.1]{bohmann2018computational}}]
Let $k$ be a field and $C$ a coalgebra spectrum that is cofibrant as a spectrum.  Then the Bousfield-Kan spectral sequence for the cosimplicial spectrum $\coTHH(C)^\bullet$ gives a coB\"okstedt spectral sequence for calculating $H_{t-s}(\coTHH(C);k)$ with $E_2$-page
\[
E_2^{s,t} = \coHH^k_{s,t}(H_*(C;k))
\]
given by the classical coHochschild homology of $H_*(C;k)$ as a graded $k$-module.
\end{theorem}

The Bousfield-Kan spectral sequence does not always converge, so the authors specify conditions under which the coB\"okstedt spectral sequence will converge completely.\footnote{For instance, the coB\"okstedt spectral sequence converges when $C$ is a suspension spectrum $\Sigma_+^\infty X$ for simply connected $X$ \cite{bohmann2018computational}.}  Note that this spectral sequence computes the ordinary homology over a field of $\coTHH(C)$ where $C$ is an $\mathbb{S}$-coalgebra.  

We want a relative version of this theorem in order to study $R$-coalgebra spectra more broadly, giving the Bousfield-Kan spectral sequence computing the homology of $\coTHH^R(C)$.  As in the $\THH$ setting, we would expect that some flatness conditions must be satisfied.
We first formally state and prove that this \textit{relative coB\"okstedt spectral sequence} exists, and then identify a corollary that will be useful for computations.  Further, we will describe conditions for convergence of this spectral sequence. Note in particular that this result holds for any generalized homology theory in addition to being over the more general ring spectrum $R$.

\begin{theorem}
Let $E$ and $R$ be commutative ring spectra, $C$ an $R$-coalgebra spectrum that is cofibrant as an $R$-module, and $N$ a $(C,C)$-bicomodule spectrum. If
$E_*(C)$ is flat over $E_*(R)$, then there exists a Bousfield-Kan spectral sequence for the cosimplicial $R$-module $\coTHH^R(N, C)^\bullet$ that abuts to $E_{t-s}(\coTHH^R(N,C))$ with $E_2$-page
\[
E_2^{s,t} = \coHH^{E_*(R)}_{s,t}(E_*(N), E_*(C))
\]
given by the classical coHochschild homology of $E_*(C)$ with coefficients in $E_*(N)$.
\end{theorem}

\begin{proof}
To begin, as in \cite{bohmann2018computational} we will recall the general construction of the Bousfield-Kan spectral sequence \cite{bousfield1972homotopy} for a general Reedy fibrant cosimplicial $R$-module $X^\bullet$.

Let $\Delta$ be the cosimplicial space with the standard $n$-simplex $\Delta^n$ as its $n^{th}$ level.  The category of $R$-modules is cotensored over pointed spaces (see e.g.\ \cite{barnes2020introduction}), and the notation $D^K$ will be used for the cotensor of an $R$-module $D$ with a simplicial space $K$.
So for a Reedy fibrant cosimplicial $R$-module $X^\bullet$ the totalization of $X^\bullet$ is given by:
\[
\Tot(X^\bullet) = eq\big(\prod_{n \ge 0} (X^n)^{\Delta^n} \rightrightarrows \prod_{\alpha \in \Delta([a], [b])} (X^b)^{\Delta^a}\big).
\]

\noindent Let $sk_s\Delta \subset \Delta$ be the cosimplicial subspace with $n^{th}$ level $sk_s\Delta^n$ that is the $s$-skeleton of the $n$-simplex.  Then one can define
\[
\Tot_s(X^\bullet) = eq\big(\prod_{n \ge 0} (X^n)^{sk_s\Delta^n} \rightrightarrows \prod_{\alpha \in \Delta([a], [b])} (X^b)^{sk_s\Delta^a}\big).
\]
The inclusions $sk_s\Delta \hookrightarrow sk_{s+1}\Delta$ then induce a tower of fibrations with fibers $F_i$:
\begin{align*}
    &\cdots \to &Tot_s(X^\bullet) &\xlongrightarrow{p_s} &Tot_{s-1}(X^\bullet) &\xlongrightarrow{p_{s-1}} &Tot_{s-2}(X^\bullet) &\to &\cdots &\xlongrightarrow{p_{1}} &Tot_0(X^\bullet) \cong X^0.\\
    & &\uparrow{i_s} & &\uparrow{i_{s-1}} & &\uparrow{i_{s-2}} & & & &\uparrow{i_0}\\
    & &F_s & &F_{s-1} & &F_{s-2} & & & &F_0
\end{align*}

\noindent We then have an associated exact couple 
\[
\begin{tikzcd}
\pi_*(\Tot_*(X^\bullet)) \arrow{rr}{p_*} && \pi_*(\Tot_*(X^\bullet))  \arrow{dl}{\partial}\\
 & \pi_*(F_*) \arrow{ul}{i_*}
\end{tikzcd}
\]
that yields a 
cohomological spectral sequence $\{E_r, d_r\}$ with differentials 
\[
d_r: E^{s,t}_r \to E^{s+r, t+r-1}_r.
\]

We now want to identify the fiber $F_s$.
Recall that the normalized complex $N^s X^\bullet$ is defined to be:
\begin{align*}
    N^sX^\bullet = \bigcap^{s-1}_{i=0} \text{ker}(\sigma_i: X^s \to X^{s-1})
\end{align*}
for codegeneracy maps $\sigma_i$ as given by the cosimplicial structure. Then each fiber $F_s$ can be identified with
\[
F_s = \Omega^s(N^s X^\bullet).
\]
Thus the $E_1$-term of the spectral sequence is given by
\begin{align*}
    E^{s,t}_1 &= \pi_{t-s}(F_s) \\
    &\cong \pi_{t-s}(\Omega^s (N^s X^\bullet))\\
    &\cong \pi_t(N^s X^\bullet)\\
    &\cong N^s \pi_t(X^\bullet)
\end{align*}
with differential $d_1 : N^s \pi_t(X^\bullet) \to N^{s+1}\pi_t(X^\bullet)$.  This map can then be identified with $\Sigma (-1)^i \pi_t(\delta^i)$ where $\delta^i$ denote the coface maps of the cosimplicial object $X^\bullet$, and we have
\begin{align*}
    &H^*(N^s \pi_t(X^\bullet)) \cong H^s(\pi_t(X^\bullet))\\
    \implies &E^{s,t}_2 \cong H^s(\pi_t(X^\bullet), \Sigma (-1)^i \pi_t(\delta^i)).
\end{align*}

Here we care about the specific case when $X^\bullet = \mathcal{R}(E \wedge \coTHH^R(N, C)^\bullet)$, where $\mathcal{R}$ indicates the Reedy fibrant replacement,
and so
\begin{align*}
    \pi_*(X^\bullet) = \pi_*(\mathcal{R}(E \wedge \coTHH^R(N, C)^\bullet)) \cong \pi_*(E \wedge \coTHH^R(N,C)^\bullet).
\end{align*}  

\noindent Recall that $\coTHH^R(N, C)$ has cosimplicial structure:
\begin{align*}
    &\vdots\\
    N \wedge_R &C \wedge_R C\\
    \uparrow \downarrow &\uparrow \downarrow \uparrow\\
    N &\wedge_R C\\
    \uparrow &\downarrow \uparrow\\
    &N
\end{align*}
so when we take $\pi_*(E \wedge - )$, at the $n^{th}$ level we see that:
\small
\begin{align*}
    \pi_*(E \wedge N \wedge_R C \wedge_R \cdots \wedge_R C)
    &\cong \pi_*(E \wedge N \wedge_{E \wedge R} E \wedge R \wedge_R C \wedge_{E \wedge R}E \wedge R \wedge_R C \cdots \wedge_{E \wedge R} E \wedge R \wedge_R C)\\
    &\cong \pi_*(E \wedge N \wedge_{E \wedge R} E \wedge C \wedge_{E \wedge R} E \wedge C \cdots \wedge_{E \wedge R} E \wedge C)\\
    &\cong \pi_*(E \wedge N) \otimes_{\pi_*(E \wedge R)} \pi_*(E \wedge C ) \otimes_{\pi_*(E \wedge R)} \cdots \otimes_{\pi_*(E \wedge R)} \pi_*(E \wedge C)\\
    &\cong E_*(N) \otimes_{E_*(R)} E_*(C) \otimes_{E_*(R)} \cdots  \otimes_{E_*(R)} E_*(C)
\end{align*}
\normalsize
\noindent where the third isomorphism follows since
    $\pi_*(E \wedge C)$ is flat over $\pi_*(E \wedge R)$ by hypothesis, and therefore
\begin{align*}
    \pi_* \mathcal{R}(E \wedge \coTHH^R(N,C)^n) &\cong \pi_* (E \wedge \coTHH^R(N,C)^n)\\
    &\cong E_*(N) \otimes_{E_*(R)} E_*(C)^{\otimes_{E_*(R)}n}.
\end{align*}
Then $\Sigma(-1)^i \pi_*(\delta^i)$ gives the coHochschild differential under this identification, and thus we get the coHochschild complex:
\begin{align*}
    E^{s,t}_2 &\cong H^s(\pi_t(X^\bullet), \Sigma (-1)^i \pi_t(\delta^i))\\
    &\cong H^s(E_t(N) \otimes_{E_*(R)} E_t(C)^{\otimes_{E_*(R)}n}, \Sigma (-1)^i \pi_t(\delta^i))\\
    &\cong \coHH^{E_*(R)}_{s,t}(E_*(N), E_*(C))
\end{align*}
Therefore the result is the Bousfield-Kan spectral sequence with $E_2$-page
\[
E_2^{s,t} = \coHH^{E_*(R)}_{s,t}(E_*(N), E_*(C))
\]
abutting to $E_{t-s}(\coTHH^R(N,C))$.
\end{proof}

Because it will be particularly useful in future examples, we state the following special case when $E=\mathbb{S}$
as a corollary:
\hypertarget{corollary}{}
\begin{corollary}\label{Cor-pi}
Let $R$ be a commutative ring spectrum and $C$ an $R$-coalgebra spectrum.
If $\pi_*(C)$ is flat over $\pi_*(R)$, then there exists a Bousfield-Kan spectral sequence that abuts to $\pi_{t-s}(\coTHH^R(C))$ with $E_2$-page
\[
E_2^{s,t} = \coHH^{\pi_*(R)}_{s,t}(\pi_*(C))
\]
given by the classical coHochschild homology of $\pi_*(C)$.
\end{corollary}

Since the Bousfield-Kan spectral sequence does not converge in general, we must specify the conditions required for convergence.  Based on work of Bousfield-Kan \cite{bousfield1972homotopy} and  Bohmann-Gerhardt-H{\o}genhaven-Shipley-Ziegenhagen \cite{bohmann2018computational}, we have the following convergence result.  

\begin{proposition}
If for every $s$ there exists some $r$ such that $E^{s,s+i}_r = E^{s,s+i}_\infty$, then the relative coB\"okstedt spectral sequence for $\coTHH^R(C)$ converges completely to
\[
\pi_* \text{Tot} \mathcal{R}(E \wedge \coTHH^R(C)^\bullet).
\]
\end{proposition}

\noindent Conditions for complete convergence can be found in Goerss-Jardine \cite{goerss2009simplicial}.  Further, from the natural construction of a map $\text{Hom}(X, Y) \wedge Z \to \text{Hom}(X, Y \wedge Z)$ we get a natural map 
\[
P: E \wedge \text{Tot}(\mathcal{R}\coTHH^R(C)^\bullet) \to \text{Tot} \mathcal{R}(E \wedge \coTHH^R(C)^\bullet),
\]
giving us the following corollary.

\begin{corollary}
If for every $s$ there exists some $r$ such that $E^{s,s+i}_r = E^{s,s+i}_\infty$ and $P: E \wedge  \text{Tot}(\mathcal{R}\coTHH^R(C)^\bullet) \to \text{Tot} \mathcal{R}(E \wedge \coTHH^R(C)^\bullet)$ induces an isomorphism in homotopy, then the relative coB\"okstedt spectral sequence for $\coTHH^R(C)$ converges completely to $E_*(\coTHH^R(C))$.
\end{corollary}

For the computations in this paper we take $E=\mathbb{S}$, and so the condition on the map $P$ is satisfied \cite{bohmann2018computational}.
We formally state this specific case here for easy reference:

\hypertarget{converge}{}
\begin{corollary}\label{converge}
When considering $E=\mathbb{S}$, if for every $s$ there exists some $r$ so that $E^{s,s+i}_r = E^{s,s+i}_\infty$ then the relative coB\"okstedt spectral sequence converges completely to $\pi_*(\coTHH^R(C))$.
\end{corollary}

\section{Algebraic structures in the (relative) (co)B\"okstedt spectral sequence}

    We will need to understand additional algebraic structure of the relative co\-B\"okstedt spectral sequence in order to facilitate calculations in the next section.  By work of Angeltveit-Rognes, the classical B\"okstedt spectral sequence for a commutative ring spectrum has the structure of a spectral sequence of Hopf algebras under certain flatness conditions \cite{angeltveit2005hopf}:  

\begin{theorem}[{\cite[Thm 4.5]{angeltveit2005hopf}}]
If $A$ is a commutative ring spectrum, then:
\begin{enumerate}
    \item If $H_*(\THH(A);\mathbb{F}_p)$ is flat over $H_*(A; \mathbb{F}_p)$, then there is a coproduct 
    \[
    \psi: H_*(\THH(A);\mathbb{F}_p) \to H_*(\THH(A);\mathbb{F}_p)  \otimes_{H_*(A;\mathbb{F}_p)} H_*(\THH(A);\mathbb{F}_p)
    \]
    and $H_*(\THH(A);\mathbb{F}_p)$ is an $\mathcal{A}_*$-comodule $H_*(A; \mathbb{F}_p)$-Hopf algebra, where $\mathcal{A}_*$ is the dual Steenrod algebra.
    \item If each term $E^r_{*,*}(A)$ for $r \ge 2$ is flat over $H_*(A; \mathbb{F}_p)$, then there is a coproduct
    \[
    \psi: E^r_{*,*}(A) \to E^r_{*,*}(A) \otimes_{H_*(A; \mathbb{F}_p)} E^r_{*,*}(A)
    \]
    and $E^r_{*,*}(A)$ is an $\mathcal{A}_*$-comodule $H_*(A; \mathbb{F}_p)$-Hopf algebra spectral sequence.  In particular, the differentials $d^r$ respect the coproduct $\psi$.
\end{enumerate}
\end{theorem}

In order to understand the implications of this spectral sequence structure in the setting of Angeltveit-Rognes \cite{angeltveit2005hopf}, we recall the following definitions of indecomposable and primitive elements.
\begin{definition}
    For an augmented algebra $A$ over a commutative ring $R$ with augmentation $\epsilon: A \to R$, the \textbf{indecomposable elements} of $A$, denoted by the $R$-module $QA$, are given by the short exact sequence
    \[
    IA \otimes IA \xlongrightarrow{\mu} IA \xlongrightarrow{} QA \xlongrightarrow{} 0
    \]
    for multiplication map $\mu$ and $IA = \text{ker}(\epsilon)$.
\end{definition}

\begin{definition}
    For a coaugmented coalgebra $C$ over a commutative ring $R$ with coaugmentation $\eta: R \to C$ and counit $\epsilon: C \to R$, the \textbf{primitive elements} of $C$, denoted by the $R$-module $PC$, are given by the short exact sequence
    \[
    0 \xlongrightarrow{} PC \xlongrightarrow{} JC \xlongrightarrow{\Delta} JC \otimes JC
    \]
    for comultiplication map $\Delta$ and $JC = \text{coker}(\eta)$. An element $x \in \text{ker}(\epsilon)$ is primitive if its image under the quotient by $Im(\eta)$ in $JC$ is in $PC$.
\end{definition}

\begin{remark}
In a coaugmented coalgebra $C$, it is equivalent to say that $x$ is primitive if $\Delta(x)=1 \otimes x + x \otimes 1$.  Note that this formulation is equivalent to the above definition because the coproduct on $x \in IC=\text{ker}(\epsilon)$ is given by
\begin{align*}
    \Delta(x) = 1 \otimes x + x \otimes 1 + \Sigma_i x'_{i} \otimes x''_{i}.
\end{align*}
Since $C$ is coaugmented, it splits as $R \oplus IC$, which means that 
\begin{align*}
    C \otimes C = (R \otimes R) \oplus (IC \otimes R) \oplus (R \otimes IC) \oplus (IC \otimes IC).
\end{align*}
Because $C$ is counital,
\[
\text{Id} = (\epsilon \otimes \text{Id}) \circ \Delta = (\text{Id} \otimes \epsilon) \circ \Delta,
\]
so $\Sigma_i x'_{i} \otimes x''_{i} \in IC \otimes IC$. But
\begin{align*}
    \eta: R &\xlongrightarrow{} C \cong R \oplus IC\\
    r &\mapsto (r,0)
\end{align*}
has cokernel $JC \cong IC$, so for primitive $x \in IC \cong JC$,
\begin{align*}
    0 \xlongrightarrow{} PC &\xlongrightarrow{} JC \xlongrightarrow{\Delta} JC \otimes JC\\
    x &\mapsto x \mapsto 0
\end{align*}
means that $\Sigma_i x'_{i} \otimes x''_{i} \in JC \otimes JC$ must be zero, and so $\Delta(x)=1 \otimes x+x \otimes 1$ as desired.
\end{remark}

If we apply these definitions to our examples of coalgebras from above, we have the following indecomposable and primitive elements, which will be useful for the calculations in the final section of this paper.

\begin{example}
Indecomposable elements in the polynomial algebra $k[w_1, w_2, \ldots]$ are classes of the form $w_i$.
The augmentation in this case is
\begin{align*}
    \epsilon: k[w_1, w_2, \ldots] &\to k\\
    w_i &\mapsto 0
\end{align*}
so $IA= \text{ker}(\epsilon) = (w_1, w_2, \ldots)$. So then the image of the product on $IA$ will be terms of the form $w_i^{m_i} \ldots w_j^{m_j}$ for $\Sigma m_k > 1$, which means $QA$ is given by elements of the form $w_i$.
\end{example}

\begin{example}
Similarly, in the exterior algebra $\Lambda_k(y_1, y_2, \ldots)$, indecomposable elements are classes of the form $y_i$. The augmentation is given by
\begin{align*}
    \epsilon: \Lambda_k(y_1, y_2, \ldots) &\to k\\
    y_i &\mapsto 0
\end{align*}
so $IA= \text{ker}(\epsilon) = (y_1, y_2, \ldots)$. The image of the product on $IA$ will be terms of the form $y_{i_1} y_{i_2} \ldots y_{i_n}$ for $n>1$, which means $QA$ is given by elements of the form $y_i$.
\end{example}

\begin{example}
Primitive elements in the polynomial coalgebra $k[w_1, w_2, \ldots]$ are classes of the form $w_i^{p^m}$ for $p = char(k)$.  The coaugmentation
\begin{align*}
    \eta: k &\xlongrightarrow{} k[w_1, w_2, \ldots]\\
    1 &\mapsto 1
\end{align*}
has cokernel $JC$ with basis $\{w_1^{j_1} w_2^{j_2} \ldots\}$ for all $j_i \ge 0$, with at least one nonzero $j_i$. 
Recall the comultiplication is given by
\begin{align*}
    \Delta(w_i^j) &= \sum_k \binom{j}{k} w_i^k \otimes w_i^{j-k}
\end{align*}
Since $p$ is the characteristic of $k$, 
\begin{align*}
    \Delta(w_i^{p^m}) = 1 \otimes w_i^{p^m} + w_i^{p^m} \otimes 1
\end{align*}
so $w_i^{p^m}$ is primitive.
The other $w_i^n$ are not primitive because $\Delta(w_i^n) \ne 1 \otimes w_i^n + w_i^n \otimes 1$ since those binomial coefficients do not vanish.
\end{example}
\begin{example}
In the exterior coalgebra $\Lambda_k(y_1, y_2, \ldots)$, primitive elements are classes of the form $y_i$. Recall that the coproduct on $\Lambda_k(y_1, y_2, \ldots)$ is given by
$\Delta(y_i)=1\otimes y_i + y_i \otimes 1$
and therefore the those terms are primitive.
\end{example}
\begin{example}
Primitive elements in the divided power coalgebra $\Gamma_k[x_1, x_2, \ldots]$ are classes of the form $x_i$. 
Recall that the divided power coalgebra has comultiplication
\begin{align*}
    \Delta(\gamma_j (x_i)) &= \sum_{a+b=j} \gamma_a(x_i) \otimes \gamma_b(x_i)
\end{align*}
So since $\gamma_0(x_i) = 1$ and $\gamma_1(x_i)=x_i$, we have
\begin{align*}
    \Delta(x_i) = 1 \otimes x_i + x_i \otimes 1.
\end{align*}
The other $\gamma_j(x_i)$ for $j>1$ are not primitive because their image under $\Delta$ will have additional $\gamma_a(x_i) \otimes \gamma_b(x_i)$ terms.
\end{example}

Studying primitive and indecomposable elements can be particularly useful because of results like the following from Angeltveit and Rognes:

\begin{theorem}[{\cite[Prop 4.8]{angeltveit2005hopf}}]

Let $A$ be a commutative $\mathbb{S}$-algebra with $H_*(A; k)$ connected and such that $\HH_*(H_*(A; k))$ is flat over $H_*(A; k)$.  Then the $E^2$-term of the B\"okstedt spectral sequence
\[
E^{2}_{*,*}(A) = \HH_*(H_*(A; k))
\] is an
$H_*(A; k)$-Hopf algebra, and a shortest non-zero differential $d^r_{s,t}$ in lowest total degree $s+t$, if one exists, must map from an algebra indecomposable to a coalgebra primitive
in $\HH_*(H_*(A; k))$.
\end{theorem}

Bohmann-Gerhardt-Shipley show that under appropriate coflatness conditions the coB\"okstedt spectral sequence for a cocommutative coalgebra spectrum has what is called a \textit{$\square$-Hopf algebra structure}, an analog of a Hopf algebra structure working over a coalgebra \cite{bohmanngerhardtshipley}, where the $\square$ in this notation is the cotensor product. Recall that
for an $R$-coalgebra $C$, a right $C$-comodule $M$ with $\gamma: M \to M \otimes C$, and a left $C$-comodule $N$ with $\psi: N \to C \otimes N$, 
the \textit{cotensor} of $M$ and $N$ over $C$ is defined to be the following equalizer in $R$-modules:

\centerline{
\xymatrix@=2cm{
& M \square_C N = eq\bigg((M \otimes_R N) \ar@<.5ex>[r]^{ \hspace{1cm}\gamma \otimes \text{Id}_N} \ar@<-.5ex>[r]_{ \hspace{1cm}\text{Id}_M \otimes \psi}  &M \otimes_R C \otimes_R N\bigg).
}
}
Note that the cotensor does not always yield a $C$-comodule, but under some conditions it does.  In particular, if $C$ is a coalgebra over a field and $M$ and $N$ are $C$-bicomodules, then $M \square_C N$ is a $C$-bicomodule. 

In order to define a $\square_C$-Hopf algebra for a coalgebra $C$ over a field $k$, we first recall the definitions of a $\square_C$-algebra, a $\square_C$-coalgebra, and a $\square_C$-bialgebra from \cite{bohmanngerhardtshipley}. See Definition 2.10, Definition 2.11, Definition 2.12, and Definition 2.13 of \cite{bohmanngerhardtshipley} for the coassociativity and counitality diagrams as well as those specifying the interactions between the algebra and coalgebra structures.

\begin{definition}[{\cite[Def 2.11]{bohmanngerhardtshipley}}]
Let $C$ be a cocommutative coalgebra over a field. A \textbf{$\square_C$-algebra} $D$ is a $C$-comodule along with a multiplication map $\mu: D \square_C D \to D$ and a unit map $\eta: C \to D$ that are associative and unital maps of $C$-comodules.
\end{definition}

\begin{definition}[{\cite[Def 2.10]{bohmanngerhardtshipley}}]
Let $C$ be a cocommutative coalgebra over a field. A \textbf{$\square_C$-coalgebra} $D$ is a $C$-comodule along with a comultiplication map $\Delta: D \to D \square_C D$ and a counit map $\epsilon: D \to C$ that are coassociative and counital maps of $C$-comodules.
\end{definition}

\begin{definition}[{\cite[Def 2.12, 2.13]{bohmanngerhardtshipley}}]
Let $C$ be a cocommutative coalgebra over a field. A \textbf{$\square_C$-bialgebra} $D$ is a $\square_C$-coalgebra that is also equipped with a multiplication map $\mu: D \square_C D \to D$ and a unit map $\eta: C \to D$ that satisfy associativity and unitality. The multiplication must also be compatible with the $\square_C$-coalgebra structure. A \textbf{$\square_C$-Hopf algebra} $D$ is a $\square_C$-bialgebra along with an antipode $\chi: D \to D$ that is a $C$-comodule map satisfying the corresponding hexagonal antipode diagram.
\end{definition}

\noindent In \cite{bohmanngerhardtshipley}, they further extend these definitions to that of a \textbf{differential bigraded $\square_C$-Hopf algebra} (Definition 6.8) and a \textbf{spectral sequence of $\square_C$-Hopf algebras} (Definition 6.9). 

We also recall the definitions of coflat and connected coalgebras. 

\begin{definition}
    For a coalgebra $C$ over a field $k$, a right comodule $M$ over $C$ is called \textbf{coflat} if $M \square_C -$ is exact as a functor from left $C$-comodules to $k$-modules.
\end{definition}

\begin{definition}
    A graded $k$-coalgebra $D_*$ is connected if $D_*=0$ when $* <0$, and the counit map $\epsilon: D_* \to k$ is an isomorphism in degree zero.
\end{definition}

\noindent We can now state Bohmann-Gerhardt-Shipley's analog of \cite[Theorem 4.5]{angeltveit2005hopf} for $\coTHH$. 

\begin{theorem}[{\cite[Thm 6.14]{bohmanngerhardtshipley}}]
For $C$ a connected cocommutative coalgebra spectrum and $k$ a field,
if for $r \ge 2$ each $E_r^{*,*}(C)$ is coflat over $H_*(C;k)$,
then the coB\"okstedt spectral sequence is a spectral sequence of $\square_{H_*(C;k)}$-Hopf algebras.
\end{theorem}

It follows from their work that the relative coB\"okstedt spectral sequence computing $\pi_*(\coTHH(C))$ for $C$ an $Hk$-coalgebra also has this type of $\square$-Hopf algebra structure. For the remainder of the paper we will focus on this case, in order to use this $\square$-coalgebra structure.

\begin{theorem}
For $C$ a cocommutative $Hk$-coalgebra spectrum,
if for $r \ge 2$ each $E_r^{*,*}(C)$ is coflat over $\pi_*(C)$, then the relative coB\"okstedt spectral sequence computing $\pi_*(\coTHH^{Hk}(C))$ is a spectral sequence of $\square_{\pi_*(C)}$-Hopf algebras.
\end{theorem}

Note that in this case the relative coB\"okstedt spectral sequence is the Bousfield-Kan spectral sequence for the cosimplicial object $\coTHH^{Hk}(C)^\bullet$, and the proof of the above result follows as in Theorem 6.14 of \cite{bohmanngerhardtshipley} for the setting of $Hk$-modules, with $\pi_*(C)$ in place of their $H_*(C;k)$ in general.

As we saw in the B\"okstedt spectral sequence, this additional algebraic structure is computationally useful. In order to understand these ideas in the dual setting, we require the following definitions and results regarding indecomposable and primitive elements.

\begin{definition}
    A unital $\square_C$-algebra $A$ with multiplication $\mu: A \square_C A \to A$ and unit $\eta: C \to A$ is \textbf{augmented} if there exists an augmentation map $\epsilon: A \to C$ such that $\epsilon \mu = \epsilon \square \epsilon$ and $\epsilon \eta = \text{Id}$.
\end{definition}

\begin{definition}
    A counital $\square_C$-coalgebra $D$ with comultiplication $\Delta: D \to D \square_C D$ and counit $\epsilon: D \to C$ is \textbf{coaugmented} if there exists a coaugmentation map $\eta: C \to D$ such that $\Delta \eta = \eta \square \eta$ and $\epsilon \eta = \text{Id}$.
\end{definition}

The definitions of primitive and indecomposable elements are then analogous to those given earlier.

\begin{definition}[{\cite[Def 2.16]{bohmanngerhardtshipley}}]
    Given a coaugmented $\square_C$-coalgebra $D$, let $PD$ be defined by the short exact sequence
    \[
    0 \xlongrightarrow{} PD \xlongrightarrow{} JD \xlongrightarrow{\Delta} JD \square_C JD,
    \]
    where $JD = \text{coker}(\eta)$.  An element in $\text{ker}(\epsilon)$ is \textbf{primitive} if its image in $JD$ is in $PD$.
\end{definition}

\begin{definition}[{\cite[Def 2.15]{bohmanngerhardtshipley}}]
    For an augmented $\square_C$-algebra $A$, the \textbf{indecomposables} of $A$, denoted by $QA$, are defined by the short exact sequence
    \[
    IA \square_C IA \xlongrightarrow{\mu} IA \xlongrightarrow{} QA \xlongrightarrow{} 0,
    \]
    where $IA = \text{ker}(\epsilon)$.
\end{definition}

With these definitions, we can now state the following result that will be particularly useful for computations.

\hypertarget{shortest}{}
\begin{theorem}\label{shortest}
For a field $k$, let $C$ be a cocommutative $Hk$-coalgebra spectrum such that $\coHH_*(\pi_*(C))$ is coflat over $\pi_*(C)$
and the graded coalgebra $\pi_*(C)$ is connected.  Then the $E_2$-term of the relative coB\"okstedt spectral sequence calculating $\pi_*(\coTHH^{Hk}(C))$,
\[
E_2^{*,*}(C) = \coHH^k_*(\pi_*(C)),
\]
is a $\square_{\pi_*(C)}$-bialgebra, and the shortest non-zero differential $d^{s,t}_r$ in lowest total degree $s+t$ maps from a $\square_{\pi_*(C)}$-algebra indecomposable to a $\square_{\pi_*(C)}$-coalgebra primitive.
\end{theorem}

The proof follows as in the non-relative version in \cite{bohmanngerhardtshipley}.

\section{Explicit calculations}

A natural question that arises when studying $\coTHH$ is to ask what kinds of coalgebra spectra exist, and for those that exist, is the $E_2$-page of the relative coB\"okstedt spectral sequence computable?  Although the coB\"okstedt spectral sequence can study examples of the form $\Sigma^\infty_{+}X$ for simply connected $X$ as in \cite{bohmann2018computational, bohmanngerhardtshipley}, we can now study general $R$-coalgebra spectra via the relative coB\"okstedt spectral sequence. Here we examine $H\mathbb{F}_p \wedge_{H\mathbb{Z}} H\mathbb{F}_p$, which is an $H\mathbb{F}_p$-coalgebra spectrum by work of Bayindir-P\'eroux \cite{bayindir2020spanierwhitehead}.
We begin by computing the $E_2$-term of the relative coB\"okstedt spectral sequence calculating $\pi_*(\coTHH^{H\mathbb{F}_p}(H\mathbb{F}_p \wedge_{H\mathbb{Z}} H\mathbb{F}_p))$.

\hypertarget{HZEX}{}
\begin{proposition}\label{HZEX}
For the $H\mathbb{F}_p$-coalgebra $H\mathbb{F}_p \wedge_{H\mathbb{Z}} H\mathbb{F}_p$, the $E_2$-page of the spectral sequence calculating $\pi_{t-s}(\coTHH^{H\mathbb{F}_p}(H\mathbb{F}_p \wedge_{H\mathbb{Z}} H\mathbb{F}_p))$ is
\[
E_2^{s,t} = \coHH^{\mathbb{F}_p}_{s,t}(\pi_*(H\mathbb{F}_p \wedge_{H\mathbb{Z}} H\mathbb{F}_p)) \cong \Lambda_{\mathbb{F}_p} (\tau) \otimes_{\mathbb{F}_p} \mathbb{F}_p[\omega]
\]
for $||\tau||=(0,1), ||\omega||=(1,1)$.
\end{proposition}

\begin{proof}
From \cite{bayindir2020spanierwhitehead}, $H\mathbb{F}_p \wedge_{H\mathbb{Z}} H\mathbb{F}_p$ is an $H\mathbb{F}_p$-coalgebra.
Note that $\pi_*(H\mathbb{F}_p \wedge_{H\mathbb{Z}} H\mathbb{F}_p)$ is flat over $\pi_*(H\mathbb{F}_p) \cong \mathbb{F}_p$ since modules are flat over fields.  Thus \hyperlink{corollary}{Corollary} \ref{Cor-pi} states that the $E_2$-page of the the relative coB\"okstedt spectral sequence computing $\pi_{t-s}(\coTHH^{H\mathbb{F}_p}(H\mathbb{F}_p \wedge_{H\mathbb{Z}} H\mathbb{F}_p))$ has the form:
\[
E_2^{s,t} = \coHH^{\mathbb{F}_p}_{s,t}(\pi_*(H\mathbb{F}_p \wedge_{H\mathbb{Z}} H\mathbb{F}_p)).
\]
We can use the K\"unneth spectral sequence to calculate $\pi_*(H\mathbb{F}_p \wedge_{H\mathbb{Z}} H\mathbb{F}_p)$ \cite[Thm IV.6.2]{elmendorf1995rings}:
\begin{align*}
    \Tor^{\pi_*(H\mathbb{Z})}_{p,q}(&\pi_*(H\mathbb{F}_p), \pi_*(H\mathbb{F}_p)) \cong \Tor^{\mathbb{Z}}_{p,q}(\mathbb{F}_p, \mathbb{F}_p)  \Rightarrow \pi_{p+q}(H\mathbb{F}_p \wedge_{H\mathbb{Z}} H\mathbb{F}_p).
\end{align*}

\noindent To compute this $E_2$-term, we form a projective resolution of $\mathbb{F}_p$ as a $\mathbb{Z}$-module:
\begin{alignat*}{4}
 &\hspace{.5cm}\mathbb{Z}\hspace{.5cm} &\xlongrightarrow{\times p} &\hspace{.5cm}\mathbb{Z}\hspace{.5cm} &\xlongrightarrow{\text{mod } p} &\hspace{.5cm}\mathbb{F}_p\hspace{.5cm} &\longrightarrow 0.
\end{alignat*}
Then we can truncate and $-\otimes_\mathbb{Z} \mathbb{F}_p$ in order to simplify the resolution to
\begin{alignat*}{3}
 &\hspace{.5cm} \mathbb{F}_p \hspace{.5cm} &\xlongrightarrow[0]{\times p} &\hspace{.5cm} \mathbb{F}_p\hspace{.5cm} &\longrightarrow 0.
\end{alignat*}
Thus we have $\mathbb{F}_p$ in degree 0 and 1, which as a coalgebra is the exterior coalgebra with a single generator in degree 1. 
Therefore the $E_2$-page looks like:
\begin{align*}
E_2^{s,t} &= \coHH^{\mathbb{F}_p}_{s,t}(\pi_*(H\mathbb{F}_p \wedge_{H\mathbb{Z}} H\mathbb{F}_p))\\
&\cong \coHH^{\mathbb{F}_p}_{s,t}(\Lambda_{\mathbb{F}_p} (\tau))\\
&\cong \Lambda_{\mathbb{F}_p} (\tau)  \otimes \mathbb{F}_p[\omega]
\end{align*}
with bidegrees $||\tau|| = (0,1)$ and $||\omega||=(1,1)$ by Proposition 5.1 in \cite{bohmann2018computational} .
\end{proof}

Such results confirm that there are 
examples for which we can compute the $E_2$-page of the relative coB\"okstedt spectral sequence calculating relative topological coHochschild homology. In fact, recent work \cite{bayindir2020spanierwhitehead} completes this computation using properties of the duality between $\coTHH$ and $\THH$ indicating that the spectral sequence should collapse in this case.

In order to apply the $\square$-Hopf algebra structure to find the $E_\infty$-page and complete the computation of the homotopy groups of $\coTHH$ in general, we will consider cocommutative coalgebras with homotopy that is similar to the above $E_2$-page example. We state further computational tools in \hyperlink{primitive}{Proposition} \ref{primitive} and \hyperlink{coalgebrass}{Theorem} \ref{coalgebrass} in order to apply them to the  computations involving the relative coB\"okstedt spectral sequence.

Recall from the previous section that the shortest nonzero differential must go from a $\square$-Hopf algebra indecomposable to a $\square$-coalgebra primitive. In later computations we will consider $\square_C$-coalgebras of the form $C \otimes D$, and Bohmann-Gerhardt-Shipley further identify primitive elements in this case:

\hypertarget{primitive}{}
\begin{proposition}[\cite{bohmanngerhardtshipley}]\label{primitive}
For cocommutative coaugmented $k$-coalgebras $C$ and $D$, $C \otimes D$ is a $\square_C$-coalgebra and an element $c \otimes d \in C \otimes D$
is primitive as an element of the $\square_C$-coalgebra $C \otimes D$ if and only if $d$ is primitive in the $k$-coalgebra $D$.
\end{proposition}

They also prove that if $\coHH(D)$ is coflat over $D$ then $\coHH(D)$ is a $\square_D$-algebra \cite{bohmanngerhardtshipley}. However, in order to similarly identify indecomposable elements, we will restrict to each specific computational setting.

Finally, for $Hk$-coalgebras the relative coB\"okstedt spectral sequence is itself a spectral sequence of $k$-coalgebras, which restricts differentials even further to targets that are $k$-coalgebra primitives.  
Bohmann-Gerhardt-H{\o}genhaven-Shipley-Ziegenhagen
\cite{bohmann2018computational} showed the following result for the ordinary coB\"okstedt spectral sequence, and the relative case follows from their work since for $E=\mathbb{S}$ we are already in the setting of cosimplicial $Hk$-modules.

\hypertarget{coalgebrass}{}
\begin{theorem}\label{coalgebrass}
If $C$ is a connected cocommutative $Hk$-coalgebra that is cofibrant as an $Hk$-module, the relative coB\"okstedt spectral sequence computing $\pi_*(\coTHH^{Hk}(C))$ is a spectral sequence of $k$-coalgebras.  In particular, for every $r>1$ there is a coproduct
\begin{align*}
    \psi: E^{*,*}_r \xlongrightarrow{} E^{*,*}_r \otimes_k E^{*,*}_r,
\end{align*}
and the differentials $d_r$ respect the coproduct.
\end{theorem}

Using these algebraic structures, we show the following result:

\begin{theorem}
For a field $k$, let $C$ be a cocommutative $Hk$-coalgebra spectrum that is cofibrant as an $Hk$-module with $\pi_*(C) \cong \Lambda_k(y)$ for $|y|$ odd and greater than 1. Then the relative coB\"okstedt spectral sequence collapses and 
\[
\pi_*(\coTHH^{Hk}(C)) \cong \Lambda_k(y) \otimes k[w]
\]
as graded $k$-modules for $|w|=|y|-1$.
\end{theorem}
\begin{proof}
The $E_2$-page of the relative coB\"okstedt spectral sequence is
\[
E_2^{s,t} = \coHH_{s,t}^k(\Lambda_k(y)) \cong \Lambda_k(y) \otimes k[w]
\]
by Proposition 5.1 in \cite{bohmann2018computational}.

By \hyperlink{shortest}{Proposition} \ref{shortest} we know that the shortest nontrivial differential in lowest total degree must map from a $\square_{\Lambda_k(y)}$-algebra indecomposable to a $\square_{\Lambda_k(y)}$-coalgebra primitive. Since the $E_2$-page is given by $\Lambda_k(y) \otimes k[w]$, \hyperlink{primitive}{Proposition} \ref{primitive} implies that elements in this $\square_{\Lambda_k(y)}$-coalgebra will be primitive if and only if the component from $k[w]$ is primitive in the $k$-coalgebra $k[w]$. Recall that primitives in the $k$-coalgebra $k[w_1, w_2, \ldots]$ more generally are of the form $w_i^{p^m}$ for $p = char(k)$, so here
\begin{align*}
    \Lambda_k(y) \otimes \big(\text{primitives in } k[w]\big) &\cong \Lambda_k(y) \otimes w^{p^m}.
\end{align*}
Therefore the only terms in the spectral sequence that are possible targets of differentials are $w^{p^m}$ and $y w^{p^m}$ for $m \ge 0$ and prime $p$. 

Bohmann-Gerhardt-Shipley also identify the indecomposable elements for the $\square_{\Lambda_k(y)}$-algebra $\Lambda_k(y) \otimes k[w]$ as those of the form $\Lambda_k(y) \otimes w$ 
\cite{bohmanngerhardtshipley}. Thus the only terms in the spectral sequence that are possible sources of differentials are $1 \otimes w = w$ and $y \otimes w = yw$. 
However, differentials from $w$ must be trivial for degree reasons. Further, because the possible targets are of the form $w^{p^m}$ and $y w^{p^m}$, and the $y w^{p^m}$ appear in the same diagonal as $yw$, those elements cannot be hit by any $(r,r-1)$-bidegree differential. Thus we need only justify why differentials from $yw$ cannot hit terms of the form $w^{p^m}$.

Note that since $||y||=(0,2n+1)$ 
and $||w||=(1,2n+1)$, the elements we are considering live in the following bidegrees ($||-||$)
for $m,n \ge 1$:
\begin{align*}
    ||w^{p^m}|| &= (p^m, p^m(2n+1)) = (p^m, 2np^m+p^m)\\
    ||y w|| &= (1, 4n+2)\\
    ||d_r(y w)|| &= (1+r, 4n+2+r-1) = (1+r, 4n+r+1).
\end{align*}

First, we will justify that $||d_r(yw)|| \ne ||w^{p^m}||$ for any $m \ge 1$.  Suppose by contradiction that these terms were in the same bidegrees. Then the first coordinate tells us that $r+1 = p^m$, so we have from the second coordinate:
\begin{align*}
    4n+p^m &= 2np^m+p^m\\
    4n &= 2np^m\\
    2 &= p^m,
\end{align*}
which is true only when $m=1$ and $p=2$.  However, if $m=1$ then $r+1=p^m=2^1$ implies that $r=1$ and we are already considering the $E_2$-page, so no such differential exists.

Note that 
the convergence conditions of \hyperlink{converge}{Corollary} \ref{converge} hold,
and the relative coB\"okstedt spectral sequence converges completely to $\pi_*(\coTHH^{Hk}(C))$.  Therefore
\[
\pi_*(\coTHH^{Hk}(C)) \cong \Lambda_k(y) \otimes k[w]
\]
as graded $k$-modules.
\end{proof}

We now generalize this result to more cogenerators. In particular, we aim to compute $\pi_*(\coTHH^{Hk}(C))$ for $C$ an $Hk$-coalgebra such that $\pi_*(C) \cong \Lambda_k(y_1, y_2)$.

\hypertarget{Exton2}{}
\begin{theorem}\label{Exton2}
Let $k$ be a field and let $p=char(k)$, including 0. For $C$ a cocommutative $Hk$-coalgebra spectrum that is cofibrant as an $Hk$-module with $\pi_*(C) \cong \Lambda_k(y_1, y_2)$ for $|y_1|, |y_2|$ both odd and greater than 1, if $p^m$ is not equal to $\frac{|y_2|-1}{|y_1|-1}+1$ for all $m \ge 1$ and $p \ne 2$,
then the relative coB\"okstedt spectral sequence collapses and
\[
\pi_*(\coTHH^{Hk}(C)) \cong \Lambda_k(y_1, y_2) \otimes k[w_1, w_2],
\]
as graded $k$-modules for $|w_i| = |y_i|-1$. Moreover, the same result holds at the prime $p=2$ with the additional assumption that $p^m \ne \frac{|y_2|-1}{|y_1|-1}$ for $m \ge 1$.
\end{theorem}

\begin{proof}
Suppose $|y_1|=a$ and $|y_2|=b$ so that on the $E_2$-page of the spectral sequence $y_1$ appears in bidegree $(0,a),$ and $y_2$ appears in bidegree $(0,b)$, which implies $||w_1|| = (1,a), ||w_2||=(1,b)$. Then we assume WLOG that $b \ge a$ and we will determine if there is the possibility for differentials by examining the degrees of the terms in the spectral sequence.

Note that because of the $\square$-coalgebra structure from \hyperlink{shortest}{Proposition} \ref{shortest} the shortest nontrivial differential has to hit a coalgebra primitive. If $char(k)=p$ a prime, then by \hyperlink{primitive}{Proposition} \ref{primitive} coalgebra primitives will be of the form
\[
\Lambda_k(y_1, y_2) \otimes w_i^{p^m}
\]
since the primitives in $k[w_1, w_2]$ are of the form $w_1^{p^m}$ or $w_2^{p^n}$.  However, by \hyperlink{coalgebrass}{Theorem} \ref{coalgebrass} the relative coB\"okstedt spectral sequence in this setting also has a coalgebra structure over $k$.  Therefore the first nontrivial differential has to hit a $k$-coalgebra primitive, that is only classes of the form $y_i$ or $w_i^{p^m}$ (and not any of their tensored combinations).  Since the $y_i$'s appear in the zero column, they cannot be hit by any differentials, so the only possible targets are classes $w_1^{p^m}$ or $w_2^{p^n}$. Similarly, if $char(k)=0$ then the only primitives in $k[w_1, w_2]$ are $w_1$ and $w_2$.

Further, the source of the shortest nontrivial differential must be a $\square$-algebra indecomposable. By \cite{bohmanngerhardtshipley}, the indecomposable elements in the $\square_{\Lambda_k(y_1, y_2)}$-algebra $\Lambda_k(y_1, y_2) \otimes k[w_1, w_2]$
will be of the form $\Lambda_k(y_1, y_2) \otimes w_i$.
Thus we only consider differentials from the following sources that land in bidegrees:
\begin{align*}
    ||d_r(w_1)|| &= (1+r, a+r-1)\\
    ||d_r(w_2)|| &= (1+r, b+r-1)\\
    ||d_r(y_1 w_1)|| &= (1+r, 2a+r-1)\\
    ||d_r(y_2 w_1)|| &= (1+r, a+b+r-1)\\
    ||d_r(y_1 w_2)|| &= (1+r, a+b+r-1)\\
    ||d_r(y_2 w_2)|| &= (1+r, 2b+r-1)\\
    ||d_r(y_1 y_2 w_1)|| &= (1+r, 2a+b+r-1)\\
    ||d_r(y_1 y_2 w_2)|| &= (1+r, a+2b+r-1).
\end{align*}    
The primitive elements that could serve as possible targets live in bidegrees:
\begin{align*}
    ||w_1^{p^m}|| &= (p^m, ap^m)\\
    ||w_2^{p^m}|| &= (p^m, bp^m).
\end{align*}

Note that if there is a nonzero differential hitting one of these classes, comparing the degree of the first coordinate implies $p^m=1+r$. In the $char(k)=0$ case, $|w_1|= (1,a)$ and $|w_2|=(1,b)$ imply that no nontrivial differentials exist since we are already on the $E_2$-page. Thus we assume $char(k)=p$ is prime and use $p^m = 1+r$ to simplify the second coordinate of the bidegree. Because differentials from $w_1$ will not hit $w_1^{p^m}$ or $w_2^{p^m}$ for degree reasons, we begin by considering differentials from $w_2$. 

Suppose that $d_r(w_2)$ hits a class $w_1^{p^m}$. Then the second coordinate gives:
\[
b+p^m-2 = ap^m,
\]
so $b-2=p^m(a-1)$, but $b-2$ is odd and $a-1$ is even so equality will not hold due to parity, so there are no such possible differentials.  
Similar parity issues arise to show $d_r(w_1) \ne w_2^{p^m}$, as well as for $d_r(w_2) \ne w_1^{p^m}$ or $w_2^{p^m}$,  $d_r(y_1 y_2 w_1) \ne w_1^{p^m}$ or $w_2^{p^m}$, and $d_r(y_1 y_2 w_2) \ne w_1^{p^m}$ or $w_2^{p^m}$.

Suppose $d_r(y_1 w_1)$ hits a class $w_1^{p^m}$. Then by comparing degrees: 
\[
2a+r-1=a(r+1).
\]
So $\frac{a-1}{a-1}=r$, but we are considering the $E_2$-page, so no such differential exists. 
A similar justification regarding $r$ determines that $d_r(y_2 w_2) \ne w_2^{p^m}$.

Now suppose $d_r(y_1 w_1)$ hits a class $w_2^{p^m}$. Then by comparing degrees: 
\[
2a+r-1=b(1+r),
\]
so $2 \frac{a-1}{b-1} = r+1$. But $a \le b$ so $2 \frac{a-1}{b-1} \le 2(1) < 3 \le r+1$ since $r \ge 2$ and so no such differential exists. Similar justifications based on the assumption that $a \le b$ allow us to conclude that $d_r(y_2 w_1) \ne w_2^{p^m}$ and $d_r(y_1 w_2) \ne w_2^{p^m}$.

Suppose $d_r(y_2 w_1)$ hits a class $w_1^{p^m}$. Then 
\[
a+ b + p^m -2 = ap^m,
\]
so $\frac{b-1}{a-1} = p^m-1$.  However we assumed that $p^m$ cannot be equal to $\frac{|y_2|-1}{|y_1|-1}+1$, so no such differential exists.  This condition also arises in the case $d_r(y_1 w_2) \ne w_1^{p^m}$ since $y_2 w_1$ and $y_1 w_2$ appear in the same bidegree.

Finally suppose $d_r(y_2 w_2)$ hits a class $w_1^{p^m}$. Then 
\[
2b+p^m-2=ap^m,
\]
so $2\frac{b-1}{a-1}=p^m$. However, if $m=0$ then $p^m=1$ and this equality does not hold since we assumed above that $b \ge a$. For $m \ge 1$, an odd prime $p$ to any power will still be odd and so $p^m \ne 2\frac{|y_2|-1}{|y_1|-1}$ due to parity. 

If $p=2$, we need only consider $m \ge 2$ since the first coordinates implied that $p^m = r+1$ and we are already on the $E_2$-page. If $\frac{|y_2|-1}{|y_1|-1}$ is odd, then $2 \frac{b-1}{a-1}$ will only be equal to a power of $p=2$ if $m=1$; if $\frac{|y_2|-1}{|y_1|-1}$ is even, then our additional assumption implies that $2^m \ne 2\frac{|y_2|-1}{|y_1|-1}$ for $m \ge 2$, so no such differential from $y_2 w_2$ to $w_1^{p^m}$ exists. 

We have now justified via combinatorics why all possible differentials can be eliminated, whether that is for parity reasons, because we are already considering the $E_2$-page, or because we restricted values of $p^m$ based on the conditions listed in the hypotheses.  Thus the spectral sequence collapses and the convergence conditions of \hyperlink{converge}{Corollary} \ref{converge} hold, so we have the desired result.
\end{proof}

Note that the conditions on $p^m$ allow us to avoid cases like $|y_1|=3, |y_2|=5$, which has a possible $d_2$ differential from $y_2 w_1$ to $w_1^3$ for the prime $p=3$ (eliminated by the primary condition) or a possible $d_3$ differential for the prime $p=2$ from $y_2 w_2$ to $w_1^4$.

Finally we state a result analogous to \cite{bohmann2018computational} for the relative coB\"okstedt spectral sequence in the case when $E = \mathbb{S}$ and $R = Hk$ for a field $k$.

\begin{theorem}
Let $C$ be a cocommutative $Hk$-coalgebra spectrum that is cofibrant as an $Hk$-module spectrum,
and whose homotopy coalgebra is
\[
\pi_*(C) = \Gamma_k[x_1, x_2, \ldots],
\]
where the $x_i$ are in non-negative even degrees and there are only finitely many of them in each degree.  Then the relative coB\"okstedt spectral sequence  calculating the homotopy groups of the topological coHochschild homology of $C$ collapses at $E_2$, and

\[
\pi_*(\coTHH^{Hk}(C)) \cong \Gamma_k[x_1, x_2, \ldots] \otimes \Lambda_k(z_1, z_2, \ldots)
\]
as $k$-modules, with $z_i$ in degree $|x_i|-1$.
\end{theorem}

\begin{proof}
Since $E_*(C) = \pi_*(C) = \Gamma_k[x_1, x_2, \ldots]$ is flat over $E_*(R) = \pi_*(Hk) \cong k$, the relative coB\"okstedt spectral sequence that abuts to $\pi_{t-s}(\coTHH^{Hk}(C))$ has $E_2$-page
\[
E_2^{s,t} = \coHH^{k}_{s,t}(\Gamma_k[x_1, x_2, \ldots]).
\]
By Proposition 5.1 in \cite{bohmann2018computational}, 
\[
\coHH^k_{*,*}(\Gamma_k[x_1, x_2, \ldots]) \cong \Gamma_k[x_1, x_2, \ldots] \otimes \Lambda_k(z_1, z_2, \ldots),
\]
where $||z_i|| = (1, |x_i|)$.  Now we want to examine the differentials on this $E_2$-page of our spectral sequence.  In particular, \hyperlink{shortest}{Theorem} \ref{shortest} says that the coalgebra structure implies that the shortest nonzero differential has to hit a $\square$-coalgebra primitive.  Since $\coHH_*(\pi_*(C))$ is a $\square$-coalgebra over $\pi_*(C)=\Gamma_k[x_1, x_2, \ldots]$, we know by \hyperlink{primitive}{Proposition} \ref{primitive} that the primitive elements will be of the form 
\[
\Gamma_k[x_1, x_2, \ldots] \otimes \big(\text{primitives in } \Lambda_k(z_1, z_2, \ldots)\big),
\]
where the primitives in $\Lambda_k(z_1, z_2, \ldots)$ viewed as a $k$-coalgebra are of the form $z_i$.

Note that since all of the $x_i$'s appear in degree $(0, |x_i|)$, all $x_i$'s and all the divided powers will stay in the zero column.  Similarly, the exterior cogenerator $z_i$ is in degree $(1, |x_i|)$, and so all possible targets, i.e. combinations of $x_i$'s with a single $z_i$, will be in the first column. Because we are on the $E_2$-page, the differentials of bidegree $(2, 1)$ will be mapping outside of these two columns, as will all possible $d_r$ differentials on later $E_r$-pages. Thus beyond the zero and first columns, the only elements that may be hit by differentials are those that include at least $z_i z_j$.  However, as we said above, such elements are not primitive, and the shortest non-zero differential $d^{s,t}_r$ in lowest total degree $s+t$ has to hit a $\square_{\pi_*(C)}$-coalgebra primitive.  Therefore, our spectral sequence collapses at $E_2$. 

Note that the convergence conditions of \hyperlink{converge}{Corollary} \ref{converge} hold,
and the relative coB\"okstedt spectral sequence converges completely to $\pi_*(\coTHH^{Hk}(C))$.
Therefore we have the following isomorphism of $k$-modules:
\[
\pi_*(\coTHH^{Hk}(C)) \cong \Gamma_k[x_1, x_2, \ldots] \otimes \Lambda_k(z_1, z_2, \ldots). \qedhere
\]
\end{proof}

\end{document}